\documentclass[12pt,oneside,a4paper]{amsart}

\usepackage{amssymb,amsmath,amsthm,hyperref}

\usepackage{stmaryrd}
\usepackage[textwidth=15cm,textheight=22cm]{geometry}

\usepackage{diagrams}
\usepackage{enumerate}

\DeclareMathOperator{\dom}{dom}   

\newcommand{\bbar}{{\ensuremath{\bar{b}}}}

\DeclareMathOperator{\rk}{rk}

\DeclareMathOperator{\dcl}{dcl}   

\DeclareMathOperator{\td}{td}  

\DeclareMathOperator{\Jac}{Jac}  
\DeclareMathOperator{\Ann}{Ann}   

\DeclareMathOperator{\Der}{Der}  

\DeclareMathOperator{\grk}{grk} 

\newcommand{\restrict}[1]{\ensuremath{\!\!\upharpoonright_{#1}}}

\newcommand{\N}{\ensuremath{\mathbb{N}}}
\newcommand{\Z}{\ensuremath{\mathbb{Z}}}
\newcommand{\Q}{\ensuremath{\mathbb{Q}}}
\newcommand{\R}{\ensuremath{\mathbb{R}}}
\newcommand{\C}{\ensuremath{\mathbb{C}}}

\newcommand{\Ran}{\ensuremath{\mathbb{R}_{\mathrm{an}}}}
\newcommand{\Rexp}{\ensuremath{\mathbb{R}_{\mathrm{exp}}}}

\newcommand{\ga}{\ensuremath{\mathbb{G}_\mathrm{a}}}  
\newcommand{\gm}{\ensuremath{\mathbb{G}_\mathrm{m}}}  

\renewcommand{\le}{\ensuremath{\leqslant}}
\renewcommand{\ge}{\ensuremath{\geqslant}}
\newcommand{\tuple}[1]{\ensuremath{\langle #1 \rangle}}
\newcommand{\class}[2]{\ensuremath{\left\{ #1 \,\left|\, #2 \right.\right\}}}

\newcommand{\subs}{\subseteq} 
\newcommand{\sups}{\supseteq} 

\newcommand{\minus}{\ensuremath{\smallsetminus}}
\newcommand{\powerset}{\ensuremath{\mathcal{P}}} 

\newcommand{\strong}{\ensuremath{\lhd}} 
\newcommand{\nstrong}{\ensuremath{\not\kern-4pt\lhd\;}} 

\newcommand{\cross}{\ensuremath{\times}}

\newbox\noforkbox \newdimen\forklinewidth
\forklinewidth=0.3pt \setbox0\hbox{$\textstyle\smile$}
\setbox1\hbox to \wd0{\hfil\vrule width \forklinewidth depth-2pt
  height 10pt \hfil}
\wd1=0 cm \setbox\noforkbox\hbox{\lower 2pt\box1\lower
2pt\box0\relax}
\def\unionstick{\mathop{\copy\noforkbox}\limits}

\def\nonfork_#1{\unionstick_{\textstyle #1}}

\setbox0\hbox{$\textstyle\smile$} \setbox1\hbox to \wd0{\hfil{\sl
/\/}\hfil} \setbox2\hbox to \wd0{\hfil\vrule height 10pt depth
-2pt width
                \forklinewidth\hfil}
\newbox\doesforkbox
\setbox\doesforkbox\hbox{\lower 2pt\box1 \lower
2pt\box2\lower2pt\box0\relax}
\def\nunionstick{\mathop{\copy\doesforkbox}\limits}

\def\fork_#1{\nunionstick_{\textstyle #1}}

\newcommand{\ra}[3]{\ensuremath{#1 \stackrel{#2}{\longrightarrow} #3}}

\newcommand{\leteq}{\mathrel{\mathop:}=}

\newtheorem{prop}{Proposition}[section]

\newtheorem{theorem}[prop]{Theorem}
\newtheorem{lemma}[prop]{Lemma}

\newtheorem{fact}[prop]{Fact}

\theoremstyle{definition}
\newtheorem{defn}[prop]{Definition}

\newcommand{\F}{\ensuremath{\mathcal{F}}}
\newcommand{\G}{\ensuremath{\mathcal{G}}}

\newcommand{\SR}{\mathrm{SR}}

\DeclareMathOperator{\hcl}{hcl}

\newcommand{\Rbar}{\ensuremath{\bar{\mathbb{R}}}}

\newcommand{\PR}{\ensuremath{\mathrm{PR}}}
\newcommand{\RPRF}{\ensuremath{\R_{\mathrm{PR}(\F)}}}

\newcommand{\ISRequivalent}{ISR-equivalent}
\newcommand{\ISR}{ISR}

\renewcommand {\d}[2] {\frac{\partial #1}{\partial #2}}

\DeclareMathOperator{\graph}{graph}

\newcommand{\Rtilde}{\tilde{\R}}

\title{Local Interdefinability of Weierstrass elliptic functions}
\author{Gareth Jones, Jonathan Kirby, Tamara Servi}
\date{version 1.0, \today}

\address{Gareth Jones \\School of Mathematics \\ University of Manchester \\ Oxford Road \\Manchester M13 9PL }
\email{gareth.jones-3@manchester.ac.uk}
\address{Jonathan Kirby\\ School of Mathematics\\ University of East Anglia\\ Norwich Research Park, Norwich NR4 7TJ}
\email{jonathan.kirby@uea.ac.uk}
\address{Tamara Servi \\ Centro de Matem\'{a}tica e Applica\c c\~oes Fundamentais  \\ Av. Prof. Gama Pinto, 2 1649-003\\ Lisboa (Portugal)}
\email{tamara.servi@gmail.com}

\begin{document}
\begin{abstract}
We explain which Weierstrass $\wp$-functions are locally definable from other $\wp$-functions and exponentiation in the context of o-minimal structures. The proofs make use of the predimension method from model theory to exploit functional transcendence theorems in a systematic way.
\end{abstract}

\maketitle

\section{Introduction}

James Ax \cite{Ax71} proved a version of Schanuel's famous conjecture in transcendental number theory, with numbers replaced by power series. Ax's theorem has been influential in model theory, particularly in work of Boris Zilber \cite{Zilber02esesc} and of the second author \cite{EAEF}. Our work here is motivated by an application of Ax's theorem to definability in the real exponential field, $\R_{\exp}$, due to Bianconi \cite{bianconi}. Using Ax's theorem in conjunction with Wilkie's model completeness result for $\R_{\exp}$ \cite{Wilkie96} (and part of the method of proof), Bianconi showed that no restriction of the sine function to an open interval is definable in $\R_{\exp}$. It is natural to ask a somewhat analogous question about definability with Weierstrass $\wp$-functions. Suppose that $\wp_1,\ldots, \wp_{N+1}$ are such functions. When is some restriction of (the real and imaginary parts of) $\wp_{N+1}$ definable in the real field expanded by suitable restrictions of (the real and imaginary parts of) $\wp_1,\ldots,\wp_N$?

In general it is not sensible to consider the definability of $\wp$-functions in their entirety, since these functions are periodic with respect to a lattice in $\C$ and hence any expansion of the real field $\Rbar$ in which a $\wp$-function is definable will interpret second-order arithmetic. Instead, we consider \emph{local definability} of a function $f$, which means definability of the restriction of $f$ to some neighbourhood of each point of its domain. Given a set $\F$ of analytic functions, there is a smallest expansion of $\Rbar$ in which all the functions in $\F$ are locally definable, which we denote $\RPRF$. Precise definitions are given in section~\ref{locdef section} below.

In this context we answer the above question completely, and also allow the exponential function to be included.
\begin{theorem}\label{interdefinability}
Let  $\F = \{\exp, f_1,\ldots,f_N\}$, where each of the $f_i$ is a Weierstrass $\wp$-function and $\exp$ is the exponential function. Let $g$ be another Weierstrass $\wp$-function. Then $g$ is locally definable in $\RPRF$ (with parameters) if and only if it can be obtained from one of $f_1,\ldots,f_N$ by isogeny and Schwartz reflection. If we omit $\exp$ from $\F$ then the exponential function is not locally definable in $\RPRF$.
\end{theorem}

Ax's theorem has been extended to include $\wp$-functions by the second author \cite{EWE,TEDESV} following work of Ax himself \cite{Ax72a}, and so one might attempt to adapt Bianconi's method to prove our theorem. This might be possible, but our efforts in this direction ran into technical difficulties due to the nature of the differential equations satisfied by the real and imaginary parts of $\wp$-functions.

Instead, we follow the method of predimensions introduced by Hrushovski \cite{Hru93} in his construction of new strongly minimal sets. As far as we are aware, this is the first application of these ideas to definability in expansions of the real field. The connection between predimensions and Schanuel conditions was made by Zilber, as part of his work on the model theory of complex exponentiation.

For the proof, we first recall a pregeometry introduced by Wilkie \cite{Wilkie08} which arises from real definability with complex functions, and we give Wilkie's characterization of this pregeometry in terms of derivations. We then introduce a predimension function and use it to show that certain derivations can be extended, following a method from \cite{EAEF}. It is here an Ax-type result first enters the picture, in an incarnation due to the second author \cite{TEDESV}. A second use of that paper together with the results on extending derivations allows us to characterize the dimension associated to the pregeometry in terms of the predimension. The theorem then follows from some computations with this predimension. In fact, our method leads naturally to a stronger result than the theorem above (Theorem~\ref{main technical result}). It is somewhat technical so we do not state it here, but when combined with some o-minimal analysis it leads to the following.
\begin{theorem}\label{intersection semialg}
Let $\F_1$ consist of complex exponentiation and some Weierstrass $\wp$-functions and let $\F_2$ consist of Weierstrass $\wp$-functions. Suppose that none of the functions in $\F_2$ is isogenous to any $\wp$-function from $\F_1$, or isogneous to the Schwartz reflection of a $\wp$-function in $\F_1$. Then any subset of $\R^n$ which is definable (with parameters) both in $\R_{\mathrm{PR}(\F_1)}$ and in $\R_{\mathrm{PR}(\F_2)}$ is semialgebraic, that is, it is definable (with parameters) in $\Rbar$.
\end{theorem}

We do not see how this sort of general result could be obtained from Bianconi's method.

Ax-style functional transcendence results have recently been shown to be very useful in applications of o-minimality to number theory, for example by Pila \cite{pilaao}. Our results could be thought of as saying that certain functions are not only algebraically independent but in fact are independent with respect to definability in a certain expansion of the real field. In the short final section we give an example showing how this sort of independence could be applied to counting certain points on certain analytic curves.

\section{Local definability}\label{locdef section}
In this brief section we explain the notion of local definability. Except where otherwise stated (such as in the statements of our main theorems), \emph{definable} means definable \emph{without parameters}.

\begin{defn}
Let $U \subs \R^n$ be an open subset and $f: U \to \R$ a function. We say that $f$ is \emph{locally definable}  with respect to an expansion $\mathcal{R}$ of $\Rbar$ if for every $a \in U$ there is a neighbourhood $U_a$ of $a$ such that $f\restrict{U_a}$ is definable.

More generally, if $M$ is a definable manifold (we will only need the cases where $M$ is affine space or projective space), then a map $f: U \to M$ is \emph{locally definable} if for each $a \in U$ there is a neighbourhood $U_a$ of $a$, an open set $W \subs M$ containing $f(U_a)$ and a definable chart $\phi: W \to \R^m$ such that each of the components of the composite $\phi \circ f\restrict{U_a}$ is definable.

In particular, identifying $\C$ with $\R^2$, a complex function is locally definable if and only if its real and imaginary parts are locally definable.
\end{defn}

\begin{defn}
Let $U \subs \R^n$ be an open subset and $f: U \to \R$ a function. A \emph{proper restriction} of $f$ is a restriction $f\restrict{\Delta}$ where $\Delta = (r_1,s_1) \cross \cdots \cross (r_n,s_n)$ is an open rectangular box with rational corners such that the closure $\bar\Delta$ of $\Delta$ is contained within $U$. If $a \in \Delta$ we say that $f\restrict{\Delta}$ is a \emph{proper restriction of $f$ around $a$}.

Given a set $\F$ of functions, we write $\mathrm{PR}(\F)$ for the set of all proper restrictions of functions in $\F$, and $\RPRF$ for the expansion of the real field $\Rbar$ by the graphs of all of the functions in $\mathrm{PR}(\F)$.
\end{defn}
We can consider \Ran\ to be the expansion of \Rbar\ by all the proper restrictions of all real-analytic functions. Usually it is defined as the expansion by all restrictions to the closed unit cube of functions which are analytic on an open neighbourhood of the cube. However, the two definitions are equivalent in the sense of giving the same definable sets.

We now establish that $\RPRF$ is indeed the smallest expansion of $\Rbar$ in which all functions from $\F$ are locally definable.
\begin{lemma}\label{local = pointwise}
A function $f: U \to \R$ is locally definable in an expansion $\mathcal{R}$ of $\Rbar$ if and only if all of its proper restrictions are definable in $\mathcal{R}$.
\end{lemma}
\begin{proof}
First suppose that all the proper restrictions of $f$ are definable and $a\in U$. Then there is a rational box $\Delta$ around $a$ whose closure is contained in $U$, so we can take $U_a = \Delta$. For the converse, let $\Delta$ be a rational box whose closure $\bar\Delta$ is contained in $U$. Then $\class{U_a \cap \bar\Delta}{a \in \bar{\Delta}}$ is an open cover of a topologically compact set, so there is a finite subcover and since $\Delta$ is definable, we see that the graph of $f\restrict\Delta$ is definable.
\end{proof}
Wilkie \cite{Wilkie08} uses definability of all proper restrictions as his definition of local definability.

\section{The holomorphic closure}

Recall that in any o-minimal expansion of $\Rbar$, in particular in the structures $\RPRF$, the definable closure $\dcl_\F$ is a pregeometry on \R. It is characterised by $b \in \dcl_\F(A)$ if and only if there is a function $f$ definable in $\RPRF$ and a tuple $a$ from $A$ such that $f(a) = b$. Since the structures $\RPRF$ have analytic cell decomposition (this follows easily from Gabrielov's theorem \cite{gabrielovcomplement}, see also Lemma \ref{analytic cell decomp} for a more general result) and we can assume that $a$ is generic in the sense of $\dcl_\F(A)$, we can take the function $f$ to be real-analytic and defined on some open subset $U$ of $\R^n$, for some $n \in \N$. Wilkie made an analogous definition in the complex case.

\begin{defn}
Given $\F$ and a subset $A \subs \C$, we define the \emph{holomorphic closure} $\hcl_\F(A)$ of $A$ by $b \in \hcl_\F(A)$ if and only if there is $n \in \N$, an open subset $U \subs \C^n$, a definable holomorphic function $f : U \to \C$, and a tuple $a \in A^n \cap U$ such that $b = f(a)$.
\end{defn}

\begin{fact}
The operator $\hcl_\F: \powerset \C \to \powerset \C$ is a pregeometry on $\C$. Furthermore, if $\F$ is countable then the holomorphic closure of a countable set is countable.
\end{fact}
\begin{proof}
The first part is from Theorem~1.10 of \cite{Wilkie08}. The observation about countability is immediate.
\end{proof}

Obviously if a holomorphic function $f: U \to \C$ is locally definable in $\RPRF$ then we have $f(a) \in \hcl_\F(a)$ for each $a \in U$. The converse is partly true.

\begin{prop}\label{gen def prop}
Let $\F$ be a countable set of holomorphic functions. Suppose $U$ is an open subset of $\C^n$ and $f: U \to \C$ is a holomorphic function such that for all $a \in U$ we have $f(a) \in \hcl_\F(a)$. Then $f$ is locally definable almost everywhere in $U$. More precisely, there is an open subset $U'$ of $U$ such that $U \minus U'$ has measure zero and $f\restrict{U'}$ is locally definable with respect to $\RPRF$. Furthermore, if $n=1$ then $U'$ can be taken such that $U \minus U'$ is a countable set.
\end{prop}
\begin{proof}
Suppose we have $f(a) \in \hcl_\F(a)$ for each $a \in U$. Enumerate all the pairs $(U_i,g_i)_{i \in \N}$ such that $U_i$ is a definable connected open subset of $U$ and $g_i :U_i \to \C$ is a definable holomorphic function. For each $a \in U$ there is $i(a) \in \N$ such that $a \in U_{i(a)}$ and $g_{i(a)}(a) = f(a)$.  Let $U'$ be the subset of $U$ consisting of those $a$ such that we can choose $i(a)$ with $g_i = f\restrict{U_i}$. Then $U'$ is open in $U$ and $f\restrict{U'}$ is locally definable.

Now let $J = \class{i \in \N}{g_i \neq f\restrict{U_i}}$. For each $i \in J$, let $V_i = \class{a \in U_i \minus U'}{g_i(a) = f(a)}$. Then $V_i$ is a proper closed subset of $U_i \minus U'$, and furthermore since it is locally the zero set of the holomorphic function $g_i - f$, it is an analytic set and has a well-defined complex dimension. Since $g_i - f$ does not vanish in the neighbourhood of any point, this dimension is strictly less than $n$, and so $V_i$ has measure zero in $U$. Thus $\bigcup_{i \in J} V_i$ has measure zero, and we note that $U' = U \minus \bigcup_{i \in J} V_i$.

If $n=1$ then the complex dimension of each $V_i$ must be $0$, so it must be a countable set. Hence $U \minus U'$ is countable.
\end{proof}

It seems not to be possible to strengthen the conclusion to get $f$ actually locally definable everywhere in $U$. While we do not have a counterexample, we do have an idea of how to produce one. Let $f:\C \to \C$ be a holomorphic function which is suitably generic, for example a Liouville function as defined by Wilkie \cite{Wilkie05}. Let $\F=\{f\}$, let $b \in \C$, set $U = \C \minus \{b\}$ and let $g = f\restrict U$. Let $h : \C \to \C$ be the constant function with value $f(b)$ and let $\G = \{g,h\}$. Then for every $a \in \C$ we have $f(a) \in \hcl_\G(a)$. However, there is no obvious way to define $f$ in a neighbourhood of $b$ and indeed we believe that if $f$ and its derivatives satisfy the transcendence property given in \cite{Zilber02tgfd} for a \emph{generic function with derivatives} then the predimension method used in this paper could be used to demonstrate that $f$ is not definable around $b$ in $\R_\G$. In this case the point $b$ is not generic in $\R_\G$. However, by making a more careful choice of functions $f$ and $h$ it seems likely we could get the same behaviour at a point $b$ which is generic in $\R_\G$.

\section{Derivations}
Let $A$ be a subfield of $\C$ and let $M$ be an $A$-vector space. Let $\F$ be a set of holomorphic functions $f: U \to \C$ where $U$ is an open subset of $\C$ which may depend on $f$.

\begin{defn}
A \emph{derivation} from $A$ to $M$ is a function $\ra{A}{\partial}{M}$ such that for each $a,b \in A$ we have
\begin{enumerate}[(i)]
\item $\partial(a+b) = \partial a + \partial b$; and
\item $\partial(ab) = a\partial b + b \partial a$.
\end{enumerate}
It is an \emph{$\F$-derivation} if and only if also for each $f \in \F$ and each $a \in A \cap \dom f$ such that $f(a) \in A$ and $f'(a) \in A$ we have $\partial f(a) = f'(a) \partial a$.
\end{defn}
(Wilkie gives the definition also when $\F$ can contain functions of several variables. We only need the 1-variable case.)

Given a subset $C \subs A$, there is an $\F$-derivation from $A$ which is universal amongst all $\F$-derivations from $A$ which vanish on $C$, written
\[\ra{A}{d}{\Omega_\F(A/C)}\]
$\Omega_\F(A/C)$ is constructed as the $A$-vector space generated by symbols $\class{da}{a\in A}$ subject to the relations which force $d$ to be an $\F$-derivation. In the case where $\F = \emptyset$ this is just the usual universal derivation $\ra{A}{d}{\Omega(A/C)}$ to the module of K\"ahler differentials. See for example \cite[Chapter~16]{Eisenbud} for more details.

We write $\Der_\F(A)$ for the $A$-vector space of all $\F$-derivations from $A$ to $A$, and $\Der_\F(A/C)$ for the subspace of all $\F$-derivations which vanish on $C$.

The connection between \F-derivations and the holomorphic closure with respect to $\F$ comes from the following result.
\begin{fact}\label{hcl derivations}
Let $A \subs \C$ and $b\in\C$. Then $b \in \hcl_\F(A)$ if and only if for every $\partial \in \Der_\F(\C/A)$ we have $\partial b = 0$.

Furthermore, $b_1,\ldots,b_n \in \C$ form an $\hcl_\F$-independent set over $A$ if and only if there are $\partial_1,\ldots,\partial_n \in \Der_\F(\C/A)$ such that
\[\partial_i b_j = \begin{cases} 1 & \mbox{ if } \quad i = j \\ 0 & \mbox{ if } \quad i \neq j. \end{cases}\]
\end{fact}
\begin{proof}
Wilkie's \cite[Theorem~1.10]{Wilkie08} states that the holomorphic closure (which he denotes by $\mathrm{LD}$) is a pregeometry on $\C$ and is identical to another operator $\mathrm{\tilde D}$. Theorem~3.4 of the same paper states that $\mathrm{\tilde D}$ is identical to an operator $\mathrm{DD}$ which is defined exactly by the condition in the right hand side of our statement. The ``furthermore'' statement follows immediately.
\end{proof}


Immediately from the universal property of $\Omega_\F(A/C)$ we see that $\Der_\F(A/C)$ is the dual vector space of $\Omega_\F(A/C)$. If $f \in \F$ and $a, f(a),f'(a) \in A$ then there is a differential form $\omega = f'(a)da - d f(a) \in \Omega(A/C)$. Letting $W$ be the span of all such $\omega$, we see that $\Omega_\F(A/C)$ is the quotient of $\Omega(A/C)$ by $W$, and that $\Der_\F(A/C)$ is the annihilator of $W$ as a subspace of $\Der(F/C)$.

So to understand $\hcl_\F$ it is enough to understand the linear relations between the differential forms associated with functions $f \in \F$. This amounts to understanding the transcendence theory for the functions $f$. In the case of the exponential function and the Weierstrass $\wp$-functions, Ax's theorem and its analogues give us sufficient understanding to obtain our main results. For general holomorphic functions we do not know so much.

\section{Weierstrass $\wp$-functions}

\subsection*{Basic properties} \ \\
We give the definition of Weiestrass $\wp$-functions and the basic properties we need following Silverman \cite[pp165--171]{Silverman09}.

Given $\omega_1, \omega_2 \in \C$ which are $\R$-linearly independent, we form the lattice $\Lambda = \class{m \omega_1 + n \omega_2}{m,n \in \Z}$. We let $\Lambda' = \Lambda \minus \{0\}$ and define the Weierstrass $\wp$-function associated with $\Lambda$ to be the meromorphic function $\wp_\Lambda(z) = \frac{1}{z^2} + \sum_{\lambda \in \Lambda'}\left(\frac{1}{(z-\lambda)^2} - \frac{1}{\lambda^2}\right)$. The poles of $\wp_\Lambda$ are precisely the elements of $\Lambda$, so there is a bijective correspondence between $\wp$-functions and lattices. It can be shown that $\wp_\Lambda(z)$ satisfies the differential equation
\[\wp_\Lambda'(z)^2 = 4\wp_\Lambda(z)^3 - g_2\wp_\Lambda(z) - g_3\]
where $g_2 = 60\sum_{\lambda \in \Lambda'} \lambda^{-4}$ and $g_3 = 140\sum_{\lambda \in \Lambda'} \lambda^{-6}$.

Let $E(\C) \subs \mathbb{P}^2(\C)$ be the complex elliptic curve given by the equation
\[Y^2Z = 4X^3 - g_2XZ^2 - g_3Z^3.\]
Then the map $\exp_E: \C \to E(\C)$ given by $z \mapsto [\wp_\Lambda(z) : \wp_\Lambda'(z) : 1]$ is a homomorphism of complex Lie groups with kernel $\Lambda$, and indeed is the universal covering map of $E(\C)$.

The multiplicative stabilizer of a lattice $\Lambda$ is the set of complex numbers $a$ such that $a \Lambda \subs \Lambda$. It is always either $\Z$ or $\Z[\tau]$ for some imaginary quadratic $\tau$ and is isomorphic to the ring of algebraic endomorphisms of the corresponding elliptic curve. We write $k_\Lambda$ or $k_E$ for the field of fractions of the multiplicative stabilizer. When $k_E \neq \Q$ then $E$ is said to have complex multiplication.

\subsection*{Use of the group structure} \ \\
Now we consider collections of holomorphic functions $\F$ in which each $f \in \F$ is either a Weierstrass $\wp$-function or the complex exponential function. Recall that $\RPRF$ is the expansion of $\bar\R$  by all proper restrictions of each function in $\F$.

The function $\wp_\Lambda : \C \to \C$ has poles exactly at the lattice points, so it is holomorphic on $\C\minus \Lambda$, and thus the open boxes we consider for proper restrictions are those whose closure does not meet $\Lambda$. However, if $\wp \in \F$ then its derivative, $\wp'$ is locally definable in $\RPRF$ by standard $\epsilon$-$\delta$ arguments, and so the map $\exp_E$ is definable on some open rectangle $\Delta$ (with Gaussian rational corners) whose closure does not meet $\Lambda$. Let $a \in \Delta$ be a Gaussian rational and let $n \in \N$. Then $\Delta' \leteq \class{n(z-a)}{z\in \Delta}$ is a rectangle about the origin, as large as we like by choosing suitable $n$. Then for $z = n(b-a) \in \Delta'$ we have $\exp_E(z) = n \cdot (\exp_E(b) - \exp_E(a))$ where the operations on the right hand side are the group operations in $E$, and so $\exp_E$ (and hence $\wp$ and $\wp'$) are locally definable as analytic functions on all of $\C$ including their poles.

Likewise, if the restriction of the complex exponential function to any open subset of $\C$ is definable then $\exp$ is locally definable on all of $\C$.



\subsection*{Isogeny and Schwartz reflection} \ \\
Next we define an equivalence relation on the set of $\wp$-functions and show that if $\wp$ is locally definable in some $\mathcal{R}$ then every function in the same equivalence class is also locally definable in $\mathcal{R}$.

Complex conjugation $z \mapsto \overline z$ is definable, so for any (locally) definable holomorphic $f$, its Schwarz reflection $f^{\SR}$ given by $f^{\SR}(z) = \overline{f(\overline z)}$ is a (locally) definable holomorphic function. The Schwarz reflection of $\wp_\Lambda$ is easily seen to be $\wp_{\overline \Lambda}$.


If $\alpha \in \C \minus \{0\}$, then an easy calculation shows that $\wp_{\alpha\Lambda}(z) = \frac{1}{\alpha^2}\wp_\Lambda(z/\alpha)$, where $\alpha\Lambda = \class{\alpha\lambda}{\lambda\in\Lambda}$. Thus if $\wp_\Lambda$ is locally definable then $\wp_{\alpha\Lambda}$ is locally definable using the parameter $\alpha$.

More generally, suppose we have lattices $\Lambda_1$ and $\Lambda_2$ with $\Lambda_1 \subs \alpha\Lambda_2$. Then we get a homomorphism between the corresponding elliptic curves $\phi_\alpha: E_1 \to E_2$ given by $\phi_\alpha = \exp_{E_2} \circ \alpha \circ \exp_{E_1}^{-1}$, which is surjective and has finite kernel. Such a homomorphism is called an \emph{isogeny}. All isogenies are rational homomorphisms, hence are definable in \Rbar. Thus $\exp_{E_2} = \phi_\alpha \circ \exp_{E_1} \circ \alpha^{-1}$, so $\exp_{E_2}$ is locally definable from $\exp_{E_1}$, and $\wp_{\Lambda_2}$ is locally definable from $\wp_{\Lambda_1}$. In general, the parameter $\alpha$ is needed.


When there is a surjective isogeny $E_1 \to E_2$, we say $E_1$ is \emph{isogenous} to $E_2$. Isogeny is an equivalence relation on elliptic curves. It gives rise to an equivalence relation on lattices with $\Lambda_1$ equivalent to $\Lambda_2$ if and only if there is $\alpha\in \C^\cross$ such that $\Lambda_1 \subs \alpha \Lambda_2$. We call this equivalence relation isogeny of lattices.

We define two Weierstrass $\wp$-functions associated with lattices $\Lambda_1$ and $\Lambda_2$ to be \emph{\ISRequivalent} (isogeny-Schwarz reflection-equivalent) if  $\Lambda_1$ is isogenous to either $\Lambda_2$ or its complex conjugate. This is also an equivalence relation. We extend our equivalence relation to the usual exponential function by saying that it forms an \ISR-class of its own.

We have shown the following, which is one direction of Theorem~\ref{interdefinability}.
\begin{prop}\label{ISR-def}
If $\wp_\Lambda$ is locally definable in an expansion $\mathcal{R}$ of $\Rbar$ and $\wp_{\Lambda'}$ can be obtained from $\wp_\Lambda$ by isogeny and Schwarz reflection then $\wp_{\Lambda'}$ is also locally definable (with parameters) in $\mathcal{R}$.
\end{prop}

\section{Predimension and strong extensions}

In this section we fix $\F = \{f_1,\ldots,f_N\}$, where the $f_i$ are $\wp$-functions or $\exp$, and are from distinct \ISR-classes. If $f_i = \wp$ we write $E_i$ for the elliptic curve corresponding to $f_i$  and $\exp_i: \C \to E_i(\C)$ for its exponential map given by $\exp_i(z) = [\wp(z) : \wp'(z) : 1]$. If $f_i = \exp$ define $E_i = \gm$ and $\exp_i = \exp$.

Let $k_i$ be $\Q$ if $E_i$ is $\gm$ or is an elliptic curve without complex multiplication, and $\Q(\tau)$ if $E_i$ is an elliptic curve with complex multiplication by $\tau$. Then the graph $\Gamma_i(\C)$ of $\exp_i$ is a subgroup of $(\ga \cross E_i)(\C)$, and it is in fact a $k_i$-vector space. Furthermore, if $A$ is any subfield of $\C$ over which all the $E_i$ are defined then $\Gamma_i(A) \leteq \Gamma_i(\C) \cap (\ga \cross E_i)(A)$ is a $k_i$-subspace of $\Gamma_i(\C)$.

\begin{defn}
Suppose that $A \subs B \subs \C$ are subfields. For each $i = 1,\ldots,N$, we define the $f_i$-group rank $\grk_i(B/A)$ to be the $k_i$-linear dimension of $\Gamma_i(B)/\Gamma_i(A)$, and we define the $\F$-group rank to be $\grk_\F(B/A) = \sum_{i=1}^N \grk_i(B/A)$.
\end{defn}

\begin{defn}
Let $A \subs B \subs \C$ be subfields with $\td(B/A)$ finite. The $\F$-predimension of $B$ over $A$ is defined as
\[\delta_\F(B/A) = \td(B/A) - \grk_\F(B/A)\]
when $\grk_\F(B/A)$ is finite, and $-\infty$ otherwise.
\end{defn}

\begin{defn}
Let $A \subs B \subs \C$ be subfields. Then we say $A$ is $\F$-strong in $B$, written $A \strong_\F B$, if and only if for all $C$ with $A \subs C \subs B$ and $\td(C/A)$ finite, we have $\delta_\F(C/A) \ge 0$.
\end{defn}

It is not clear from the definition that $\F$-strong proper subfields of $\C$ exist, but later in Proposition~\ref{closed implies strong} we will show that they do, in fact that every subfield which is closed in the sense of the pregeometry $\hcl_\F$ is $\F$-strong. Note that since $\F$ is finite, there are $\hcl_\F$-closed proper subfields of $\C$, indeed $\hcl_\F(\emptyset)$ is countable.

In fact there are many more $\F$-strong subfields than $\hcl_\F$-closed subfields.
\begin{lemma}\label{chain lemma}
Let $A \strong_\F B \strong_\F \C$ be $\F$-strong subfields. Then there is an ordinal $\lambda$ and a chain of subfields $(A_\alpha)_{\alpha \le \lambda}$ such that $A_0 = A$ and $A_\lambda = B$ and:
\begin{enumerate}[(i)]
\item For each $\alpha < \lambda$, either $\td(A_{\alpha+1}/A_\alpha) = 1$ and $\delta(A_{\alpha+1}/A_\alpha) = 1$ or $\td(A_{\alpha+1}/A_\alpha)$ is finite and $\delta(A_{\alpha+1}/A_\alpha) = 0$.
\item If $\alpha$ is a limit then $A_\alpha = \bigcup_{\beta < \alpha} A_\beta$.
\item If $0 \le \alpha \le \beta \le \lambda$ then $A_\alpha \strong_\F A_\beta$.
\end{enumerate}
\end{lemma}
\begin{proof}
Enumerate $B$ as $(b_\alpha)_{\alpha < \lambda}$ for some limit ordinal $\lambda$. Assume inductively that we have $A_\beta$ for $\beta < \alpha$ satisfying the conditions (i)--(iii) and such that $A_\beta \strong_\F B$ and $b_\gamma \in A_\beta$ whenever $\gamma < \beta$. If $\alpha$ is a limit, take $A_\alpha = \bigcup_{\beta < \alpha} A_\beta$.

Now suppose $\alpha$ is a successor, say $\alpha = \gamma +1$. If there is a finite transcendence degree extension $F$ of $A_\gamma$ containing $b_\gamma$ such that $\delta(F/A_\gamma) = 0$ then choose $A_\alpha$ be some some such $F$. Otherwise, take $A_\alpha$ to be the algebraic closure of $A_\gamma \cup\{b_\gamma\}$. Then $\delta(A_\alpha/A_\gamma) = \td(A_\alpha/A_\gamma) = 1$. Conditions (i) and (ii) are immediate, (iii) is straightforward to verify, and $B = A_\lambda$ by construction.
\end{proof}

\subsection*{Extending derivations}

\begin{prop}\label{delta = 0 extension}
Suppose that $A \strong_\F B$ are subfields of $\C$ with $\td(B/A)$ finite and $\delta_\F(B/A) = 0$, and let $\partial \in \Der_\F(A)$. Then $\partial$ extends uniquely to a derivation $\partial' \in \Der_\F(B)$.
\end{prop}
\begin{proof}
Let $\partial \in \Der_\F(A)$. Let $\Omega(B/\partial)$ be the quotient of $\Omega(B)$ by the relations $\sum a_i db_i = 0$ for those $a_i, b_i \in A$ such that $\sum a_i \partial b_i = 0$. Then we have quotient maps of $B$-vector spaces
\[\begin{diagram} \Omega(B) &\rOnto &\Omega(B/\partial)& \rOnto & \Omega(B/A) \end{diagram}.\]

Let $\Der(B/\partial) = \class{\eta \in \Der(B)}{(\exists \lambda \in B)[\eta\restrict{A} = \lambda\partial]}$. Then $\Der(B/\partial)$ is a $B$-vector subspace of $\Der(B)$ that is easily seen to be the dual space of $\Omega(B/\partial)$. Thus we have a sequence of inclusions
\[\begin{diagram} \Der(B/A) &\rInto & \Der(B/\partial)& \rInto & \Der(B) \end{diagram}\]
dual to the sequence above.

Since $A \strong_\F B$ we have $\sum_{i=1}^N \grk_i(B/A) \le \td(B/A)$, finite. For each $i=1,\ldots, N$ let $n_i = \grk_i(B/A)$, choose $(b_{i,j}, \exp_i(b_{i,j}))$ for $j=1,\ldots,n_i$, forming a $k_i$-linear basis for $\Gamma_i(B)$ over $\Gamma_i(A)$. Then for each $i$ and $j$ we have differential forms
$\omega_{ij} = f_i'(b_{i,j}) db_{i,j} - df_i(b_{i,j}) \in \Omega(B)$, and their images $\hat{\omega}_{ij} \in \Omega(B/A)$.

Let $G$ be the algebraic group $\prod_{i=1}^N E_i^{n_i}$, let $n = \sum_{i=1}^N n_i$ and let $TG = \ga^n \cross G$. Write $\bbar$ for the tuple of all $b_{i,j}$, and $\exp_G(\bbar) \in G(B)$ for the tuple of all $\exp_i(b_{i,j})$, so $(\bbar, \exp_G(\bbar)) \in TG(B)$. Since the $E_i$ are non-isogenous, every algebraic subgroup of $G$ is of the form $\prod_{i=1}^N H_i$ where each $H_i$ is a subgroup of $E_i^{n_i}$. Furthermore, these $H_i$ are given by $k_i$-linear equations. Since the $b_{i,j}$ are chosen so that the pairs $(b_{i,j}, \exp_i(b_{i,j}))$ are $k_i$-linearly independent in $\Gamma_i(B)$ over $\Gamma_i(A)$, it follows that the point $(\bbar,\exp_G(\bbar))$ does not lie in any coset $\gamma \cdot TH$ where $H$ is a proper algebraic subgroup of $G$ and $\gamma$ is defined over $A$. Hence, by \cite[Proposition~3.7]{TEDESV}, the differential forms $\hat\omega_{ij}$ are all $B$-linearly independent in $\Omega(B/A)$. It follows that the images of the $\omega_{ij}$ in $\Omega(B/\partial)$ are $B$-linearly independent. Thus their span, say $W$, has dimension equal to $\grk_\F(B/A)$.

Thus the annihilator $\Ann(W)$ has codimension $\grk_\F(B/A)$ in $\Der(B/\partial)$ and also in $\Der(B/A)$. Observe that for $\eta \in \Der(B/A)$, we have $\eta \in \Der_\F(B/A)$ if and only if $\eta \in \Ann(W)$. We have
\begin{eqnarray*}
\dim \Der_\F(B/A) &=& \dim \Der(B/A) - \dim W \\
 & = & \td(B/A) - \grk_\F(B/A) \\
 & = & \delta_\F(B/A) = 0
\end{eqnarray*}
so $\Der_\F(B/A) = \{0\}$.

 If $\partial = 0$ then $\Der_\F(B/\partial) = \Der_\F(B/A)$ and so $\partial$ only extends to the zero derivative on $B$. Otherwise $\partial \neq 0$ and
\[\dim \Der(B/\partial) = \dim \Der(B/A) + 1 = \td(B/A) + 1.\]
Again for $\eta \in \Der(B/\partial)$ we have $\eta \in \Der_\F(B/\partial)$ if and only if $\eta \in \Ann(W)$. So we have
\begin{eqnarray*}
\dim \Der_\F(B/\partial) &=& \dim \Der(B/\partial) - \dim W \\
 & = & \td(B/A) +1 - \grk_\F(B/A) \\
 & = & \delta_\F(B/A) +1 = 1 .
\end{eqnarray*}
Thus there is $\eta \in \Der_\F(B/\partial) \minus \Der_\F(B/A)$, unique up to scalar multiplication. So $\eta\restrict{A} = \lambda\partial$ for some $\lambda\neq 0$, and thus $\lambda^{-1}\eta$ is the unique $\F$-derivation on $B$ which extends $\partial$, as required.
\end{proof}

\begin{prop}\label{simple extension}
Suppose that $A \strong_\F B$ are subfields of $\C$ with $\td(B/A) = \delta_\F(B/A) = 1$, let $b \in B \minus A$, and let $\partial \in \Der_\F(A)$. Then for each $c \in B$, the derivation $\partial$ extends uniquely to a derivation $\partial' \in \Der_\F(B)$ such that $\partial' b = c$. In particular there is $\partial' \in \Der_\F(B/A)$ such that $\partial' b = 1$.
\end{prop}
\begin{proof}
The derivation $\partial \in \Der_\F(A)$ extends uniquely to a field derivation $\partial' \in \Der(B)$ with $\partial' b = c$. Since $\td(B/A) = \delta_\F(B/A) = 1$ we have $\grk_\F(B/A) = 0$, so for each $i$ we have $\Gamma_i(B) = \Gamma_i(A)$, and thus $\partial' \in \Der_\F(B)$.
\end{proof}

\begin{prop}\label{extend derivations}
Let $A \strong_\F \C$ be a subfield and let $\partial \in \Der_\F(A)$. Then there is $\partial' \in \Der_\F(\C)$, extending $\partial$.
\end{prop}
\begin{proof}
Put together Lemma~\ref{chain lemma}, Proposition~\ref{delta = 0 extension} and Proposition~\ref{simple extension}.
\end{proof}

\subsection*{Predimension and dimension}\

\begin{prop}\label{closed implies strong}
Let $C \subs \C$ be an $\hcl_\F$-closed subfield. Then $C \strong_\F \C$.
\end{prop}
\begin{proof}
Let $A \subs C$ be a subfield of $\C$, with $\td(A/C)$ finite. Suppose we have $x_{i,j} \in A$ for $i=1,\ldots,N$, $j  = 1,\ldots,n_i$ such that, for each $i$, the pairs $(x_{i,j},\exp_i(x_{i,j}))$ lie in $\Gamma_i(A)$ and are $k_i$-linearly independent over $\Gamma_i(C)$. Let $n = \sum_{i=1}^N n_i$, let $S = \prod_{i=1}^N E_i^{n_i}$ and let $TS = \ga^n \cross S$. Let $x$ be the tuple of all the $x_{i,j}$ and let $y$ be the tuple of all the $\exp_i(x_{i,j})$. Then in the notation of \cite{TEDESV} we have $(x,y) \in \Gamma_S(A) \subs TS(A)$.

Since the groups $E_i$ are pairwise non-isogenous, every algebraic subgroup of $S$ is of the form $\prod_{i=1}^N H_i$, where $H_i$ is an algebraic subgroup of $E_i^{n_i}$. Thus the $k_i$-linear independence condition implies that $(x,y)$ does not lie in any coset $\gamma \cdot TH$ where $H$ is a proper algebraic subgroup of $S$ and $\gamma$ is a point of $TH$ defined over $C$.

Thus, by \cite[Theorem~3.8]{TEDESV} with $F = \C$ and $\Delta = \Der_\F(\C/C)$, we have
\[\td(x,y/C) - \rk \Jac(x,y) \ge n.\]
The quantity $\rk \Jac(x,y)$ is the rank of a matrix, hence non-negative, so we have
\[\td(A/C) \ge \td(x,y/C) \ge n.\]
Hence $n$ is bounded above so we may choose the $x_{i,j}$ to make it maximal. Then $n = \grk_\F(A/C)$ and we conclude that
\[\delta_\F(A/C) = \td(A/C) - \grk_\F(A/C) \ge 0.\]
Hence $C \strong_\F \C$, as required.
\end{proof}

We write $\dim_\F$ for the dimension with respect to the pregeometry $\hcl_\F$.
\begin{prop}\label{predim dim}
Let $C$ be an $\hcl_\F$-closed subfield of $\C$, and let $A \subs \C$ be an extension of $C$ with $\td(A/C)$ finite. Then:
\begin{enumerate}[(i)]
\item $\dim_\F(A/C) = \min \class{\delta_\F(B/C)}{A \subs B \subs \C \mbox{ with }\td(B/A) \mbox{ finite}}$;
\item if $A \strong_\F \C$ then $\dim_\F(A/C) = \delta_\F(A/C)$.
\end{enumerate}
\end{prop}
\begin{proof}
For (i), since $C \strong_\F \C$, the minimum exists and is non-negative. If $B$ witnesses the minimum then $B \strong_\F \C$, and so it remains to show (ii).

So suppose $A \strong_\F \C$. From Lemma~\ref{chain lemma} there is a chain
\[ C = A_0 \strong_\F A_1 \strong_\F \cdots \strong_\F A_r = A\]
with $\delta(A_i/A_{i-1}) = 0$ or $\delta(A_i/A_{i-1}) = \td(A_i/A_{i-1}) = 1$ for each $i = 1,\ldots,r$.

If $\delta(A_i/A_{i-1}) = 1$ then by Proposition~\ref{simple extension} we have $\dim \Der_\F(A_i/A_{i-1}) = 1$. Since $A_i \strong_\F \C$, these derivations all extend to $\C$, and hence by Fact~\ref{hcl derivations} we have $\dim_\F(A_i/A_{i-1}) = 1$. If $\delta(A_i/A_{i-1}) = 0$ then from Proposition~\ref{extend derivations} it follows that $\dim  \Der_\F(A_i/A_{i-1}) = 0$ and hence $\dim_\F(A_i/A_{i-1}) = 0$.

Now $\delta_\F(A/C) = \sum_{i=1}^r \delta_\F(A_i/A_{i-1})$ and $\dim_\F(A/C) = \sum_{i=1}^r \dim_\F(A_i/A_{i-1})$, so we have
$\dim_\F(A/C) = \delta_\F(A/C)$ as required.
\end{proof}

\section{Definability of other functions}

\begin{lemma}
Suppose $C \strong_\F \C$ is an $\F$-strong subfield and $A, B$ are extensions of $C$ of finite transcendence degree. We write $AB$ to mean the subfield of $\C$ generated by $A \cup B$. Then:
\begin{enumerate}
\item $\grk_\F(AB/C) + \grk_\F(A \cap B/C) \ge \grk_\F(A/C) + \grk_\F(B/C)$ (we say $\grk$ is upper semi-modular)
\item $\td(AB/C) + \td(A\cap B/C) \le \td(A/C) + \td(B/C)$ (transcendence degree is lower semi-modular, or submodular)
\item $\delta(AB/C) + \delta(A\cap B/C) \le \delta(A/C) + \delta(B/C)$ (predimension is submodular)
\item If $A \subs B$ and $\F_1 \subs \F_2$ then $\grk_{\F_1}(A/C) \le \grk_{\F_2}(B/C)$ (monotonicity in $A$ and $\F$).
\end{enumerate}
\end{lemma}
\begin{proof}
Straightforward.
\end{proof}

Now we can prove our main technical result.
\begin{theorem}\label{main technical result}
Let $\F_1$ and $\F_2$ be sets of Weierstrass $\wp$-functions, with one or both possibly also containing complex exponentiation. Let $\F_0 = \F_1 \cap \F_2$ and suppose that no $\wp$-function from $\F_1 \minus \F_0$ is isogenous to any $\wp$-function in $\F_2$, or to the Schwarz reflection of a $\wp$-function in $\F_2$.

Suppose that $f: U \to \C$ is a holomorphic function which is locally definable (with parameters) with respect to $\R_{\PR(\F_1)}$ and with respect to $\R_{\PR(\F_2)}$. Then $f$ is locally definable (with parameters) almost everywhere in $U$, with respect to $\R_{\PR(\F_0)}$, in the sense of \ref{gen def prop}.
\end{theorem}

\begin{proof}
Let $\F_3 = \F_1 \cup \F_2$. We may assume that $\F_3$ is finite and contains at most one representative of each \ISR-class. Let $C$ be a countable subfield of $\C$ which is $\hcl_{\F_3}$-closed and contains the parameters needed to define $f$. It follows that $C$ is $\hcl_{\F_i}$-closed for all $i=0,1,2,3$.

Let $a \in U$. Then for $i=1,2$ we have $f(a) \in \hcl_{\F_i}(C a)$, so $\dim_{\F_i}(f(a)/Ca) = 0$. Let $d_i = \dim_{\F_i}(a/C)$ for $i=0,1,2,3$. By Proposition~\ref{predim dim} there are subfields $B_1$ and $B_2$ of $\C$ containing $C \cup \{a,f(a)\}$ such that $\delta_{\F_i}(B_i/C) = d_i$. We choose the $B_i$ to be the smallest such fields, so $B_i \strong_\F \C$, and we let $A = B_1 \cap B_2$ and let $B$ be the subfield of \C\ generated by $B_1 \cup B_2$.

Let $\F_4 = \F_1\minus \F_0$ and $\F_5 = \F_2 \minus \F_0$, so $\F_0$, $\F_4$ and $\F_5$ are disjoint. Henceforth we write $\grk_i$ for $\grk_{\F_i}$, $\delta_i$ for $\delta_{\F_i}$, and $\dim_i$ for $\dim_{\F_i}$.

Then we have:
\begin{eqnarray*}
\grk_3(B/C) &=& \grk_4(B/C) + \grk_5(B/C) + \grk_0(B/C) \\
& \ge & \grk_4(B_1/C) + \grk_5(B_2/C) +  \grk_0(B/C)  \mbox{ by monotonicity}\\
& \ge & \grk_4(B_1/C) + \grk_5(B_2/C) + \grk_0(B_1/C) + \grk_0(B_2/C) - \grk_0(A/C)\\
&&  \mbox{ by upper semi-modularity}\\
& = &  \grk_1(B_1/C) + \grk_2(B_2/C) - \grk_0(A/C). \\
\end{eqnarray*}

Also $\td(B/C) \le  \td(B_1/C) + \td(B_2/C) - \td(A/C)$ by submodularity, so:
\begin{align*}
d_3 &\le  \delta_3(B/C) \\
&=  \td(B/C) -\grk_3(B/C) \\
& \le \td(B_1/C) + \td(B_2/C) - \td(A/C)  \\
& \hspace{1cm} - [\grk_1(B_1/C) + \grk_2(B_2/C) - \grk_0(A/C)]\\
& \le \delta_1(B_1/C) + \delta_2(B_2/C) - \delta_0(A/C)\\
\end{align*}
and so $\delta_0(A/C) \le d_1 + d_2 - d_3$.
Now we have $d_3 \le \min(d_1,d_2)$ since $\F_3 \sups \F_1, \F_2$ and so $\delta_0(A/C) \le d_3$.  So $\dim_0(a,f(a)/C) \le d_3$ by Proposition~\ref{predim dim} again, but since $\F_0 \subs \F_3$ we have $\dim_0(a/C) \ge d_3$ and hence $\dim_0(f(a)/C a) = 0$. So $f(a) \in \hcl_{\F_0} (Ca)$.
This applies to all $a \in U$, and hence, by Proposition~\ref{gen def prop}, $f$ is locally definable (with parameters from $C$) almost everywhere in $U$, with respect to $\R_{\PR(\F_0)}$.
\end{proof}

We can now complete the proof of Theorem~\ref{interdefinability}.
\begin{proof}[Proof of Theorem~\ref{interdefinability}]
Proposition~\ref{ISR-def} shows that if $g$ is ISR-equivalent to one of the $f_i$ then it is locally definable in $\RPRF$. So suppose $g$ is not ISR-equivalent to any of the $f_i$. By Proposition~\ref{ISR-def}, we may suppose the $f_i$ are all from different \ISR-classes. Applying Proposition~\ref{main technical result} with $\F_1 = \F$ and $\F_2 = \{g\}$, we see that $g$ is not locally definable in $\RPRF$.
\end{proof}

\section{Definable sets}

In the case where $\F_0 = \emptyset$, we are able to extend our result from holomorphic functions to all definable sets. This is because we can show that all definable sets come in some way from definable holomorphic functions.

%
\begin{prop}\label{holomorphic extensions} Suppose that $\F$ is a collection of holomorphic functions. If $f:U\to \R$ is an analytic function definable in the structure $\RPRF$, on $U$, an open subset of $\R^n$, then there exists an $\RPRF$-definable subset $X$ of $U$ of dimension at most $n-1$, an open subset $W$ of $\C^n$ with $W\cap \R^n = U\setminus X$ and an $\RPRF$-definable holomorphic $F:W\to \C$ such that $F\restriction_{U\setminus X} = f$.
\end{prop}
\begin{proof}
 For the result we will quote we need our functions to be total, so we first fix for each open box in $\R^m$ and each $m$ a semialgebraic analytic isomorphism between the box and $\R^m$. Let $\Rtilde$ be the expansion of the real field by all real and imaginary parts of proper restrictions of functions in $\F$ suitably composed with these isomorphisms to make total functions, and then add all derivatives of these functions. Note that the definable sets in $\Rtilde$ and in $\RPRF$ are the same. For this proof, we write definable to mean definable in either of these structures. Suppose that $F\in \F$ and that $\phi$ and $\psi$ are the real and imaginary parts of some proper restriction of $F$. For convenience, suppose that the box these functions are restricted to is $[-1,1]^{2n}$. Define holomorphic functions $F_\phi,F_\psi$ near $0$ in $\C^{2n}$ by
 \begin{eqnarray*}
 F_{\phi}(z,w) &=& \frac{F(z+iw)+\overline{F(\overline{z}+i \overline{w})}}{2} \\
 F_{\psi}(z,w)&=& \frac{-i\cdot F(z+iw)+\overline{-i\cdot F(\overline{z}+i \overline{w})}}{2}
 \end{eqnarray*}
 where the bars denote coordinatewise complex conjugation. These functions, near $0$, extend $\phi$ and $\psi$ respectively and are definable. Using compactness, we see that the real and imaginary parts of any proper restriction of functions in $\F$ have definable holomorphic extensions, and so the same is true for the functions in the language of $\Rtilde$ and so for all terms in the language of $\Rtilde$ (after these are identified with functions).

 By a theorem of Gabrielov \cite[Theorem~1]{gabrielovcomplement} the structure $\Rtilde$ is model complete. It is also polynomially bounded, and hence it is, in the sense of \cite{jw} \emph{locally polynomially bounded} (we do not need the definition here). So we can apply Corollary 4.5 of \cite{jw} to our function $f$ to obtain  definable open sets $U_1,\ldots, U_k$ such that the dimension of $U\setminus \bigcup U_i$ is at most $n-1$ and such that for each $i$ there exist terms $g_{i,1},\ldots,g_{i,m_i}:\R^{n+m_i}\to \R$ in the language of $\Rtilde$, with $m_i\ge 1$ and definable functions $\phi_{i,1},\ldots,\phi_{i,m_i}:U_i\to \R$ such that
\begin{eqnarray*}
g_{i,1}(x,\phi_{i,1},\ldots,\phi_{i,m_i}(x)) &=&0 \\
&\vdots &\\
g_{i,m_i}(x,\phi_{i,1},\ldots,\phi_{i,m_i}(x))&=&0\\
\end{eqnarray*}
and
\begin{equation*}
\det \begin{pmatrix}
\d{g_{i,1}}{x_{n+1}} & \cdots &  \d{g_{i,1}}{x_{n+m_i}} \\
\vdots & \vdots &\vdots \\
\d{g_{i,m_i}}{x_{n+1}} & \cdots &  \d{g_{i,m_i}}{x_{n+m_i}}
\end{pmatrix}
(x) \ne 0
\end{equation*}
for all $x \in U_i$ and $\phi_{i,1}=f\restriction_{U_i}$. Let $X=U\setminus \bigcup U_i$. We prove the conclusion for each restriction $f\restriction_{U_i}$ and for simplicity of notation we drop the $i$ and suppose that the above holds throughout $U$. So we have terms $g_1,\ldots,g_m :\R^{n+m}\to \R$ and definable functions $\phi_1,\ldots,\phi_m:U\to \R$ such that $f=\phi_1$ and the equations and inequation above hold on $U$. As we noted above, the terms $g_1,\ldots,g_m$ all have definable holomorphic extensions. So we obtain definable holomorphic functions $G_1,\ldots,G_m$ on a definable open set $V$ in $\C^{n+m}$ with $V\cap\R^{n+m}=\R^{n+m}$. The graph of $f$ is contained in the common zero set $Z$ of $G_1,\ldots,G_m$ and at each point of this graph the hypotheses of the complex implicit function theorem hold. For each $x\in U$ let $\epsilon_x>0$ be such that the complex implicit function theorem holds on  $W_x:= \{ z\in V: |z-x|< \epsilon_x\}$ so that $Z$ is the graph of a map above $W_x$. By definable choice we can take $\epsilon$ to be definable function of $x$ and then the union $W$  of the $W_x$ is a definable open set. The projection of $Z$ to $\C^{n+1}$ intersected with $W\times \C$  is the graph of a holomorphic extension of $f$ on $W$, and is definable.
\end{proof}

We need a cell decomposition result which may be known, though it is not well-known.
\begin{lemma}\label{analytic cell decomp}
Suppose that $\mathcal R$ is an expansion of $\Rbar$ in which every definable set is also definable in $\Ran$. Then $\mathcal R$ has analytic cell decomposition.
\end{lemma}
\begin{proof}
Throughout this proof, by definable
we mean definable with parameters in $\mathcal{R}$.
It suffices to show that every definable cell $C\subseteq\mathbb{R}^{n}$
is a disjoint union of finitely many definable analytic cells. This
is proven by induction on the pairs $\left(n,k\right)$, where $k$
is the dimension of $C$. Clearly, we may suppose that $n>1$ and
$k>0$.

Suppose first that $C=\text{graph}\left(f\right)$, where $f:C'\to\mathbb{R}$
is definable and $C'\subseteq\mathbb{R}^{n-1}$ is a definable cell
of dimension $k$. By the inductive hypothesis,
we may suppose that $C'$ is an analytic cell. So
we can find a definable analytic diffeomorphism $\varphi:C'\to C''$
onto an open analytic cell $C''\subseteq\mathbb{R}^{k}$, with analytic
inverse $\psi:C''\to C'$. By Tamm's Theorem (see \cite[Theorem 2.3.3]{tamm}
and also the formulation given in \cite[p. 1367]{vdd:mill:tamm}),
there exists $p\in\mathbb{N}$ such that, if $U\subseteq C''$ is
an open set and $f\circ\varphi:C''\to\mathbb{R}$ is $\mathcal{C}^{p}$
on $U$, then it is actually analytic on $U$. Consider a $\mathcal{C}^{p}$-cell
decomposition of $C''$ and let $D_{1},\ldots,D_{l}$ be the open
cells. Then $\psi\left(D_{i}\right)\subseteq C'$ is an analytic cell
for each $i=1,\ldots,l$ and so is $\text{graph}\left(f\restriction_{\psi\left(D_{i}\right)}\right)\subseteq C$.
Notice that $\text{dim}\left(C\setminus\bigcup_{i=1}^{l}\psi\left(D_{i}\right)\right)<k$,
hence by the inductive hypothesis on $k$ and by further cell decomposition,
the set $C\setminus\bigcup_{i=1}^{l}\psi\left(D_{i}\right)$ has an
analytic cell decomposition.

Suppose now that $C=\left\{ \left(x,y\right):\ x\in C',\ f\left(x\right)<y<g\left(x\right)\right\} $,
where $f,g:C'\to\mathbb{R}$ are definable and $C'\subseteq\mathbb{R}^{n-1}$
is a definable cell of dimension $k-1$. By the first part of this
proof, there are analytic cell decompositions $\mathcal{D}_{f}$ and
$\mathcal{D}_{g}$ of $\text{graph}\left(f\right)$ and $\text{graph}\left(g\right)$,
respectively. Let $\mathcal{D}'$ be an analytic cell decomposition
of $C'$ which is compatible with every set $\pi\left(D\right)$,
where $D\in\mathcal{D}_{f}\cup\mathcal{D}_{g}$ and $\pi:\mathbb{R}^{n}\ni\left(x,y\right)\mapsto x\in\mathbb{R}^{n-1}$
is the projection onto the first $n-1$ coordinates. By the inductive
hypothesis on $k$, it suffices to remark that, by construction, for
every cell $D\in\mathcal{D}'$ of dimension $k-1$, the set $C\cap\left(D\times\mathbb{R}\right)$
is a disjoint union of finitely many analytic cells.
\end{proof}

The following lemma is proved by inserting the word `analytic' at appropriate points in the proof of \cite[Lemma 1.1 page 32]{wilkiecovering}.
\begin{lemma}\label{retraction}
Suppose that $C \subs \R^n$ is an analytic cell definable in an o-minimal expansion of \Rbar. Then there is an open analytic cell $U \subs \R^n$ such that $C \subs U$ and an analytic retraction $\theta: U \to C$ (that is, for each $a \in U$ we have $\theta(a) = a$) which are definable in the structure $\tuple{\Rbar,C}$.
\end{lemma}

We can now prove Theorem~\ref{intersection semialg}.

\begin{proof}[Proof of Theorem~\ref{intersection semialg}]
Suppose that $X \subs \R^n$ is definable in $\R_{\F_1}$ and in $\R_{\F_2}$. By Lemma~\ref{analytic cell decomp} there is an analytic cell decomposition of $X$ in the structure $\tuple{\Rbar;X}$, which is then a decomposition in both $\R_{\F_1}$ and $\R_{\F_2}$. Hence it suffices to prove the theorem in the case where $X$ is an analytic cell, $C \subs \R^n$.

We proceed by induction on $n$, with the case $n=1$ being trivial. So suppose $n > 1$.

If $C = \graph(f)$ with $f:C' \to \R$ then, by induction, $C'$ is a semialgebraic cell. Let $\theta: U' \to C'$ be the analytic retraction provided by Lemma~\ref{retraction}, and note that it is definable in $\R_{\F_1}$ and $\R_{\F_2}$. The function $f \circ \theta : U' \to \R$ is analytic and definable in  $\R_{\F_1}$ and $\R_{\F_2}$. Hence using Proposition~\ref{holomorphic extensions} its holomorphic extension $g$ to some open set $W \subs \C^n$ is  definable both in $\R_{\F_1}$ and in $\R_{\F_2}$. But then, by Theorem~\ref{main technical result}, some restriction of $g$ to an open set is definable in $\Rbar$, so is a semialgebraic function. But $g$ is holomorphic so it must be an algebraic function, and hence $f \circ \theta$ and $f$ are also algebraic functions. Hence $C$ is a semialgebraic cell.

Otherwise $C$ is a parametrized interval of the form $(f,g)_{C'}$. By the previous case, $f$ and $g$ are semialgebraic functions on $C'$, so $C$ is also semialgebraic.
\end{proof}

Note that in this last result we have worked with parameters. Care needs to be taken when formulating an analogue for $0$-definable sets. For example, we could take $\mathcal{F}_1$ to consist of the exponential on its own, and $\F_2$ to consist of the $\wp$-function associated to the curve $Y^2=4X^3-e$, where $e=\exp(1)$. Then $e$ is $0$-definable in both $\R_{\F_1}$ and in $\R_{\F_2}$. But it is certainly not $0$-definable in $\Rbar$! So at the very least, the conclusion would be that a set that is $0$-definable in both structures is $0$-definable in the expansion of the real field by the constants occurring in the equations for all the curves associated to the $\wp$-functions in $\F_1$ and $\F_2$. Perhaps this is the correct formulation, but it seems hopeless to prove as it makes assertions about relations between values of iterated exponentials and $\wp$-functions. For example, take $\F_1$ to again consist of the exponential only and $\F_2$ to consist of a single $\wp$-function now associated to a curve defined by an equation over the rationals. Then it may be that $e^e$ and $\wp(\wp(\wp(1)))$ are transcendental but interalgebraic, and in that case $e^e$ would be a 0-definable point in both $\R_{\F_1}$ and $\R_{\F_2}$, but not 0-definable in $\Rbar$. Ruling out all such coincidences is part of the content of Bertolin's \emph{Conjecture Elliptico-Torique} \cite{Bertolin02}, which is a special case of the Andr\'e-Grothendieck conjecture on the periods of 1-motives.

\section{A Counting Application}
Our main result can be used to show that certain functions are transcendental. For example, suppose that $\F$ consists of all $\wp$-functions, that  $f:(a,b)\to (a',b')$ is an analytic homeomorphism definable in $\R_{\exp}$ and that $g:(a',b')\to (a',b')$ is a transcendental function definable in $\R_{\F}$. Then the function $h:(a,b)\to (a,b)$ defined by
$$
h(t)=f^{-1}(g(f(t)))
$$
is transcendental. For otherwise, we would have $h$ algebraic and then $g= f\circ h\circ f^{-1}$ would be definable in $\Rexp$, which would contradict Theorem~\ref{main technical result}.

Using this together with recent results on counting rational points on definable sets in o-minimal structures (that is, results starting with the Pila-Wilkie theorem \cite{pw}) we can count points of the form $f(q)$ on certain sets, in terms of the height of $q$. We give an example in the direction of transcendence theory. First, recall that if a rational $q$ is written in lowest terms as $a/b$ then the height $H(q)$ is defined as the maximum of $|a|$ and $|b|$, and that this height function is extended to tuples by taking coordinatewise maximum. Now suppose that, as above, we take $\F$ to be the collection of all $\wp$-functions. We take $f$ above to be the logarithm and so count points of the form $(\log p,\log q)$ on the graph of an $\R_{\F}$ definable function $g:\R\to\R$, in terms of the height of $(p,q)$. This is clearly the same as counting rational points $(p,q)$ on the graph of the function $h(t)=\exp (g(\log t))$. This function is definable in the expansion $\R_{\F,\exp}$ of $\R_{\F}$ by the exponential function. Adding the Pfaffian chains of the functions in the language of $\R_{\F}$ (see \cite{Macintyrepfun}) and constants to the language of $\R_{\F,\exp}$, we obtain a reduct of the expansion of the real field by all Pfaffian functions and this reduct is model complete by a result of van den Dries and Miller \cite[Corollary 6.12(ii)]{vdDMillerWilkie}. By the application of Theorem~\ref{main technical result} described above, the function $h$ is transcendental. Hence we can apply a result due to the first author and Thomas \cite[Theorem 4.4]{jt} to obtain the following.
\begin{prop} Suppose that $g:\R\to\R$ is definable in the expansion of the real field by all proper restrictions of the real and imaginary parts of all Weierstrass $\wp$-functions. Then there exist positive $c$ and $k$ depending on $g$ with the following property. There are at most
$$
c(\log H)^k
$$
positive rationals $p$ and $q$ of height at most $H$ such that $(\log p,\log q)$ lies on the graph of $g$.
\end{prop}

Taking another $f$ in place of the logarithm, or interchanging the role of the structures, leads to other similar results.
\section*{Acknowledgements}
We thank Sergei Starchenko for a useful conversation. We also thank the Max Planck Institute for Mathematics, Bonn and the Centro di Ricerca Matematica Ennio De Giorgi,
Pisa, where some of this work was carried out.

Gareth Jones was supported by EPSRC grants EP/F043236/1, EP/E050441/1 and EP/J01933X/1.

Tamara Servi was supported by by FCT PEst OE/MAT/UI0209/2011, by FCT Project PTDC/MAT/122844/2010 and by FIRB 2010 ``New advances in Model Theory of exponentiation''.

\bibliographystyle{alpha}



\end{document}